\documentclass{article}

\usepackage[a4paper]{geometry}
\usepackage{amsmath}
\usepackage{amssymb}
\usepackage{amsthm}
\usepackage{caption2}
\usepackage{multirow}
\usepackage{enumitem}
\usepackage{bbm}

\newtheorem{theorem}{Theorem}[section]
\newtheorem{lemma}[theorem]{Lemma}

\newtheorem{corollary}[theorem]{Corollary}

\newtheorem{definition}[theorem]{Definition}

\newtheorem{remark}[theorem]{Remark}

\title{}

\begin{document}
\title{A note on the Tur\'an number for the traces of hypergraphs}

\author{Bingchen Qian$^{\text{a,b}}$ and Gennian Ge$^{\text{a,}}$\thanks{
  Corresponding author. Email address: gnge@zju.edu.cn. Research supported by National Key Research and Development Program of China under Grant Nos. 2020YFA0712100 and
2018YFA0704703, the National Natural Science Foundation of China under Grant No. 11971325, and Beijing Scholars Program.}\\
\footnotesize $^{\text{a}}$ School of Mathematical Sciences, Capital Normal University, Beijing, 100048, China\\
\footnotesize $^{\text{b}}$ School of Mathematical Sciences, Zhejiang University, Hangzhou 310027, Zhejiang, China}

\maketitle

\begin{abstract}
Let $\mathcal{H}$ be an $r$-uniform hypergraph and $F$ be a graph. We say $\mathcal{H}$ contains $F$ as a trace if there exists some set $S\subseteq V(\mathcal{H})$ such that $\mathcal{H}|_{S}:=\{E\cap S: E\in E(\mathcal{H})\}$ contains a subgraph isomorphic to $F.$ Let $ex_r(n,Tr(F))$ denote the maximum number of edges of an $n$-vertex $r$-uniform hypergraph $\mathcal{H}$ which does not contain $F$ as a trace. In this paper, we improve the lower bounds of $ex_r(n,Tr(F))$ when $F$ is a star, and give some optimal cases. We also improve the upper bound for the case when $\mathcal{H}$ is $3$-uniform and $F$ is $K_{2,t}$ when $t$ is small.


\smallskip


\end{abstract}

\section{Introduction}
A hypergraph $\mathcal{H}$ is $r$-uniform if it is a family of $r$-element subsets of a finite vertex set.  We usually denote its vertex set and edge set by $V(\mathcal{H})$ and $E(\mathcal{H})$, respectively.

\begin{definition}
   For a graph $F$ with vertex set $\{v_1,\ldots, v_p\}$ and edge set $\{e_1,\ldots , e_q\},$ a hypergraph
$\mathcal{H}$ contains $F$ as a trace if there exists a set of distinct vertices $W:=\{w_1, \ldots, w_p\} \subseteq V(\mathcal{H})$ and distinct edges $\{f_1,\ldots , f_q\} \subseteq E(\mathcal{H}),$ such that if $e_i = v_\alpha v_\beta,$ then $\{w_\alpha, w_\beta\} = f_i\cap W.$ The vertices $\{w_1, \ldots, w_p\}$ are called the base vertices of $F.$
\end{definition}

For $r\ge2,$ let $ex_r(n,Tr(F))$ be the maximum number of edges of an $n$-vertex $r$-uniform hypergraph that does not contain $F$ as a trace.

In \cite{MR2365419}, Mubayi and Zhao determined the asymptotic value of $ex_r(n, Tr(K_s))$ for all $r$ when $s\in\{3,4\}$ and conjectured that for $s\ge 5,$ $ex_r(n,Tr(K_s))\sim \left(\frac{n}{s-1}\right)^{s-1}.$ Later, Sali and Spiro \cite{MR3690260} determined the order of magnitude of $ex_r(n, Tr(K_{s,t}))$ when $t\ge (s-1)!+1,$ $s\ge 2r-1.$  Recently, F\"uredi and Luo \cite{2020arXiv200207350F} deduced the order of magnitude of $ex_r(n,Tr(F))$ for all graphs $F$ in terms of their generalized Tur\'an numbers. In particular, they showed
$$ex_r(n,Tr(F)) = \Theta\left(\max_{2\le s\le r}ex(n, K_s, F)\right),$$
where $ex(n,K_s,F)$ denotes the maximum number of copies of $K_s$ in an $F$-free graph on $n$ vertices. In the case where $F$ is outerplanar, they gave the following theorem.
\begin{theorem}[\cite{2020arXiv200207350F}]
  If $F$ is an $m$-vertex outerplanar graph, then
  $$ex(n-r+2,F)\le ex_r(n,Tr(F))\le \frac{1}{2}r^r(m-2)^{(r-2)}ex(n,F).$$
\end{theorem}

In \cite{2020arXiv200207350F}, F\"uredi and Luo also investigated the case when $F$ is a star and gave both lower and upper bounds. Later, Luo and Spiro \cite{MR4245297} gave an upper bound on the number of edges in a $3$-uniform hypergraph which is $Tr(K_{2,t})$-free for $t\ge 14.$

In this note, we first improve the lower bounds for $ex_r(n,Tr(K_{1,t}))$ and show that our lower bounds are optimal for infinitely many cases. Secondly, by adapting the method used in \cite{MR4245297}, we improve the upper bound for $ex_3(n,Tr(K_{2,t}))$ when $t$ is small.

The paper is organised as follows. In Sections $2$ and $3,$ we study the lower bounds for $ex_r(n, K_{1,t})$ and the upper bounds for $ex_3(n,Tr(K_{2,t})),$ respectively. We conclude in Section $4$.

\section{Improved lower bounds for $ex_r(n,Tr(K_{1,t}))$}
We first give the definition of the well-known structure called disjointly representable family.
\begin{definition}
  The sets $E_1, E_2, \ldots, E_k\subseteq X$ are said to be disjointly representable if there exist $x_1,x_2,\ldots,x_k\in X$ such that
  \begin{equation*}
    x_i\in E_i \Leftrightarrow i=j~ (1\le i,j \le k),
  \end{equation*}
  in other words, no $E_i$ is contained in the union of the others.
\end{definition}

Let $f(r,k)=\max|\mathcal{H}|,$ where the maximum is taken over all $r$-uniform set-systems $\mathcal{H}$ containing no $k$ disjointly representable members. In the literature, the best upper and lower bounds for $f(r,k)$ were given by Frankl and Pach \cite{MR739411}.

\begin{theorem}[\cite{MR739411}]\label{thm of tight upper bound}
  We have
  \begin{equation*}
    f(r,k)\le {r+k-1 \choose k-1}, \text{where } r\ge1, k\ge2,
  \end{equation*}
  \begin{equation*}
    f(r,3)=\left\lfloor\frac{r+2}{2}\right\rfloor\left\lceil\frac{r+2}{2}\right\rceil, \text{where } r\ge1,
  \end{equation*}
  and (up to isomorphism) the only $r$-uniform set-system $\mathcal{H}_{r,3}$ without $3$ disjointly representable members, satisfying $\mathcal{H}_{r,3}=f(r,3),$ can be constructed as follows. Let $A$ and $B$ be disjoint set with $|A|=\lfloor\frac{r+2}{2}\rfloor$ and $|B|=\lceil\frac{r+2}{2}\rceil,$ and let $\mathcal{H}_{r,3}:=\{E\subseteq A\cup B \mid |E|=r \text{ and } |A\backslash E|=|B\backslash E| =1\}.$
  \begin{equation*}
    f(2,k)={k+1\choose 2}-\left\lceil\frac{k+1}{2}\right\rceil, \text{ where }  k\ge2,
  \end{equation*}
  and the only graph $\mathcal{H}_{2,k}$ with $|\mathcal{H}_{2,k}|=f(2,k),$ which contains no $k$ disjointly representable edges, can be obtained from the complete graph $K_{k+1}$ by deleting $\lceil\frac{k+1}{2}\rceil$ edges as disjoint as possible.
\end{theorem}

In \cite{2020arXiv200207350F}, F\"uredi and Luo gave a connection between $f(r,k)$ and the upper bound of $ex_r(n,Tr(K_{1,t})),$ and gave the following theorem.
\begin{theorem}[\cite{2020arXiv200207350F}]\label{upper and lower bounds for ex(n,Tr(K))}
  For any $r\ge2,t\ge2,$ if $n=a(r+t-2)+b$ with $b\le r+t-3$ then
  \begin{equation*}
    a{r+t-2\choose r}+{b\choose r}\le ex_r(n,Tr(K_{1,t}))\le \frac{n}{r}f(r-1,t)\le\frac{n}{r}{r+t-2\choose r-1}.
  \end{equation*}
  In particular, if $n$ is divisible by $r+t-2,$ the lower bound is $\frac{n}{r}{r+t-3\choose r-1}.$
\end{theorem}



Before we improve the lower bound, we first give the following definition.
\begin{definition}
  Let $v\ge k\ge t$ and let $X$ be a set of $v$ elements (points). A $t$-$(v,k,\lambda)$ covering is a collection of $k$-subsets (blocks) of $X,$ denoted by $\mathcal{C}$, such that every $t$-subset of points occurs in at least $\lambda$ blocks in $\mathcal{C}.$ The size of a covering is just the size of $\mathcal{C}.$
\end{definition}


Now we improve the lower bound.
\begin{theorem}\label{thm new lower bound by covering}
  For any $r\ge2,t\ge2,$ if $n=a(r+t-1)+b$ with $b\le r+t-2,$ and there exists a minimal $(r-1)$-$(r+t-1,r,1)$ covering with size $c,$  then
  \begin{equation*}
    a\left({r+t-1\choose r}-c\right)+{b\choose r}\le ex_r\left(n,Tr(K_{1,t})\right).
  \end{equation*}
\end{theorem}
\begin{proof}
  Let each component of $\mathcal{H}$ be a clique of size $r+t-1$ such that there are as many cliques as possible. For each clique, we delete a minimal $(r-1)$-$(r+t-1,r,1)$ covering. And the left vertices just form a clique. If $n=a(r+t-1)+b$ with $b\le r+t-2,$ then $e(\mathcal{H})=a\left({r+t-1\choose r}-c\right)+{b\choose r}$. It only needs to show that $\mathcal{H}$ is $Tr(K_{1,t})$-free.
  \begin{enumerate}
    \item If the component with vertex size $b\le r+t-2$ contains a $Tr(K_{1,t})$ with base vertices $\{v_1,\ldots,v_{t+1}\},$ then each edge must contain at least $3$ base vertices, a contradiction.
    \item If the component with vertex size $r+t-1$ contains a $Tr(K_{1,t})$ with vertices $\{v_1,\ldots,v_{t+1}\}$ and edges $\{v_1v_2, v_1v_3,\ldots, v_1v_{t+1}\}.$ By the definition of $Tr(K_{1,t}),$ the corresponding hyperedges of $\mathcal{H}$ must share the left common $r-2$ vertices $\{v_{t+2},\ldots,v_{r+t-1}\}$. And these $t$ hyperedges are all the edges containing $\{v_1,v_{t+2},\ldots,v_{r+t-1}\}$ in the clique, which is a contradiction since we have deleted an $(r-1)$-$(r+t-1,r,1)$ covering.
    \end{enumerate}
\end{proof}
Now, we give some exact results for the values of $ex_r(n,Tr(K_{1,t})).$

\begin{theorem}
  For the case $r=3$,  the following holds.
  \begin{enumerate}
    \item If $t+2 \equiv 0 \mod 6$ and $t+2\mid n,$ then $ex_3(n,Tr(K_{1,t}))=\frac{n}{6}(t^2-2).$
    \item If $t+2 \equiv 1 \text{ or } 3 \mod 6$ and $t+2\mid n,$ then $ex_3(n,Tr(K_{1,t}))=\frac{n}{6}(t^2-1).$
  \end{enumerate}
\end{theorem}
\begin{proof}
We first prove the upper bounds. By Theorem \ref{thm of tight upper bound} and Theorem \ref{upper and lower bounds for ex(n,Tr(K))}, we know that $ex_3(n,Tr(K_{1,t}))\le \frac{n}{2}f(2,t)=\frac{n}{3}\left({t+1\choose 2}-\lceil\frac{t+1}{2}\rceil\right).$

\begin{enumerate}
  \item If $t+2 \equiv 0$ $\mod 6$, then $t$ is even, and $ex_3(n,Tr(K_{1,t}))\le \frac{n}{3}\left({t+1\choose 2}-\lceil\frac{t+1}{2}\rceil\right)=\frac{n}{3}(\frac{t^2+t}{2}-\frac{t}{2}-1)=\frac{n}{6}(t^2-2).$
  \item If $t+2 \equiv 1 \text{ or } 3$ $\mod 6$, then $t$ is odd, and $ex_3(n,Tr(K_{1,t}))\le \frac{n}{3}\left({t+1\choose 2}-\lceil\frac{t+1}{2}\rceil\right)=\frac{n}{3}\left({t+1\choose 2}-\frac{t+1}{2}\right)=\frac{n}{6}(t^2-1).$
\end{enumerate}

   Now we prove the lower bounds. If $t+2 \equiv 0$ $\mod 6,$ it is known (see \cite{MR2246267}) that there exists a $2$-$(t+2,3,1)$ covering of size $\frac{(t+2)^2}{6},$ so when $t+2\mid n,$ the lower bound follows from Theorem \ref{thm new lower bound by covering}. If $t+2 \equiv 1$ or $3$ $\mod 6,$ the well-known Steiner Triple System $(t+2,3,1)$ exists, thus if $t+2\mid n,$ then the lower bound also follows from Theorem \ref{thm new lower bound by covering}.
\end{proof}

\begin{theorem}
  When $r=2k$ and $2k(k+1)\mid n,$ we have that $ex_r(n,Tr(K_{1,3}))=\frac{n(k+1)}{2}.$
\end{theorem}
\begin{proof}
  Assume that $r,k$ and $n$ are integers satisfying the condition of the theorem. $ex_r(n,Tr(K_{1,3}))\le\frac{n(k+1)}{2}$ follows from Theorem \ref{thm of tight upper bound} and Theorem \ref{upper and lower bounds for ex(n,Tr(K))}.

  Tur\'an's theorem (see \cite{MR2246267}) says that suppose $q=\lfloor\frac{n'}{n'-h-1}\rfloor,$ then there exists an $h$-$(n',n'-2,1)$ covering of size $qn'-{q+1\choose 2}(n'-h-1).$ In this case $n'=r+2, h=r-1, q=\lfloor\frac{r+2}{2}\rfloor=k+1,$ and there exists an $(r-1)$-$(r+2,r,1)$ covering of size $(k+1)(2k+2)-2{k+2\choose2}=k(k+1).$ Thus $ex(n,Tr(K_{1,3}))\ge\frac{n(k+1)}{2}$ follows from Theorem \ref{thm new lower bound by covering}.
\end{proof}

\section{Improved upper bounds for $ex_3(n,K_{2,t})$ when $t$ is small}
In this section, we focus on $3$-uniform hypergraphs.
In \cite{MR4245297}, Luo and Spiro gave the following theorem.
\begin{theorem}[\cite{MR4245297}]\label{Luo for K 2,t}
  For $t\ge 14,$
  $$
  ex_3(n,Tr(K_{2,t}))\le\frac{1}{6}\left(t^{3/2}+55t\sqrt{\log t}\right)n^{3/2} + o(n^{3/2}).
  $$
\end{theorem}

In this section, we generalise the result of Luo and Spiro for $ex_3(n,Tr(C_4))$ in \cite{MR4245297} by a similar analysis, and obtain an improved upper bound for $ex_3(n,Tr(K_{2,t}))$ when $t$ is small.

Given a hypergraph $\mathcal{H}$, we define $d_\mathcal{H}(x,y)$ to be the number of edges of $\mathcal{H}$ containing $\{x,y\}$, and we call this number the co-degree of $\{x,y\}$. We denote $A$ the set of edges containing at least one pair of vertices with co-degree $1.$ We will often identify hypergraphs by their set of edges and write e.g. $\mathcal{H}\backslash A$ to denote the hypergraph $\mathcal{H}$ after deleting some set of edges $A$ from $E(H)$.

\begin{lemma}[\cite{MR4245297}]\label{lemma K2,t free max codegree}
  Let $t\ge2$ and let $H$ be a $Tr_3(K_{2,t})$-free $3$-uniform hypergraph on $[n].$ For any pair $\{x,y\},$ we have $d_{\mathcal{H}\backslash A}(x,y)\le 3t-3.$
\end{lemma}

Following the notations in \cite{MR4245297}, for $v\in V(\mathcal{H}),$ we define the $1$ and $2$-neighborhood of $v$ as
$$N_1(v)=\{x\in V(\mathcal{H}): \exists e \in E(\mathcal{H}), \{v,x\}\subseteq e\}.$$
$$N_2(v)=\{x\in V(\mathcal{H})\backslash(N_1(v)\cup\{v\}):\exists h\in E(\mathcal{H}),x\in e, e\cap N_1(v)\neq\emptyset\}.$$
That is, $N_i(v)$ is the set of vertices that are at distance $i$ from $v.$

When the co-degrees of all pairs of vertices in $V(\mathcal{H})$ are at most $k$, it is observed in \cite{MR4245297} that if $E$ is a set of edges containing some vertex $v$ and $V$ is the set of vertices $u\neq v$ with $u\in e$ for some $e\in E,$ then
\begin{equation}\label{equation v and e}
  |V|\ge \frac{2}{k}|E|,
\end{equation}
as each vertex in $V$ is contained in at most $k$ edges with $v$.

Let $\mathcal{H}$ be an $n$-vertex $3$-uniform hypergraph with no $K_{2,t}$ trace for $t\ge 3$. Let $A$ be the edges with at least one pair of co-degree $1,$ and $B=\mathcal{H}\backslash A.$ It was shown in \cite{MR4245297} that
$$|A|\le ex(n, K_{2,t})\le \frac{\sqrt{t-1}}{2}n^{3/2}+o(n^{3/2}).$$

Now we focus on the size of $B.$ From now on, we write $d_B(v),d_B(u,v)$ as $d(v),d(u,v).$ By Lemma \ref{lemma K2,t free max codegree}, $d(x,y)\le 3t-3$ for all $\{x,y\}\subseteq V(\mathcal{H}).$ For any vertex $v$ we also let $N_1(v)$ and $V_2(v)$ denote the $1$ and $2$-neighborhoods of $v$ in $B$, respectively.

\begin{lemma}\label{lemma of the size of two neighbor intersection}
  For any $x,y\in V(\mathcal{H}),$ $|N_1(x)\cap N_1(y)|\le (t-1)(6t-2).$
\end{lemma}
\begin{proof}
  Suppose that there exist $x,y\in V(\mathcal{H})$ and some set $\{u_1,u_2,\ldots,u_{(t-1)(6t-2)+1}\}\subseteq N_1(x)\cap N_1(y).$ At most $3t-3$ $u_i$'s, say $\{u_{(t-1)(6t-5)+2},u_{(t-1)(6t-5)+3},\ldots,u_{(t-1)(6t-2)+1}\}$, are in edges of the form $\{x,y,u_i\}\in B.$ Let $G$ be a graph on $[(t-1)(6t-5)+1]$ where $ij\in E(G)$ if and only if either $\{x,u_i,u_j\}\in B$ or $\{y,u_i,u_j\}\in B$. Because pairs in $B$ have co-degree at most $3t-3,$ we have that $\Delta(G)\le6t-6.$ By greedy algorithm, we can always find a $\overline{K_t}$, say $\{u_1,u_2,\ldots,u_{t}\}.$ Let $e(x,i)$ and $e(y,i)$ be the edges of $B$ containing $\{x,u_i\}$ and $\{y,u_i\}$ for $1\le i\le t$, respectively. Note that both $y,u_j\notin e(x,u_i)$ and $x,u_j\notin e(y,u_i)$ hold for $i\neq j.$ Thus, these $2t$ edges form a $K_{2,t}$ trace in $B$, which is a contradiction.
\end{proof}

Fix any vertex $v\in V(\mathcal{H}).$ We define $E_u=\{e\in B: e\cap N_1(v)=\{u\}\}$ and $V_u=\{w\in N_2(v): \exists e\in E_u,w\in e\}$. Since $V_u\subseteq N_1(u)$ for all $u,$ we have the following corollaries.
\begin{corollary}
  Let $e=\{v,u,w\}\in B$ be any edge containing $v.$ Then $|V_u\cap V_w|\le (t-1)(6t-2).$
\end{corollary}
\begin{proof}
  It follows directly from Lemma \ref{lemma of the size of two neighbor intersection}.
\end{proof}

\begin{corollary}
  For all $u\in N_1(v),$
  $$|V_u|\ge\frac{2}{3t-3}\left(d(u)-(3t-3)-3(t-1)^2(6t-2)\right).$$
\end{corollary}
\begin{proof}
  Note that $E_u$ consists of every edge containing $u$ except for at most $3t-3$ edges which also contain $v$ and the edges $\{e\in B: u\in e, |e\cap N_1(v)|\ge2\}.$ We claim that the latter set has cardinality at most $3(t-1)^2(6t-2)$. Indeed, any such edge would contribute a vertex to $N_1(u)\cap N_1(v),$ of which there are at most $(t-1)(6t-2)$ vertices by the previous lemma. Each of such vertices can be contained in at most $3t-3$ edges together with $u$ because $B$ has maximum co-degree at most $3t-3.$ Therefore, we have that $|E_u|\ge d(u)-(3t-3)-3(t-1)^2(6t-2).$

   Since $B$ has maximum co-degree at most $3t-3,$ and by (\ref{equation v and e}), we have that $|V_u|\ge\frac{2}{3t-3}|E_u|.$ This completes the proof.
\end{proof}

\begin{lemma}
  $\sum_{u\in N_1(v)}|V_u|\le (t-1)(6t-2)n.$
\end{lemma}
\begin{proof}
  If $|N_1(v)|\le(t-1)(6t-2),$ then the conclusion follows directly since $|V_u|\le n.$ Now we assume that $|N_1(v)|>(t-1)(6t-2),$ if for any distinct $(t-1)(6t-2)+1$ vertices $u_1,\ldots,u_{(t-1)(6t-2)+1}\in N_1(v),$ there exists a vertex $w\in V_{u_1}\cap\cdots\cap V_{u_{(t-1)(6t-2)+1}},$ then we have that $|N_1(v)\cap N_1(w)|>(t-1)(6t-2)+1,$  which contradicts to Lemma \ref{lemma of the size of two neighbor intersection}. Thus, any vertex can be contained in at most $(t-1)(6t-2)$ different $V_u's,$ the conclusion follows by double counting.
\end{proof}

  By the discussions above, we have that
  $$\sum_{u\in N_1(v)}d(u)\le\sum_{u\in N_1(v)}\left(\frac{3t-3}{2}|V_u|+f(t)\right)\le3(t-1)^2(3t-1)n+2d(v)f(t),$$
  where $f(t)=3t-3+3(t-1)^2(6t-2),$ and the last inequality comes from the fact that $|N_1(v)|\le 2d(v).$

  Let $d=3e(\mathcal{H})/n$ be the average degree of $\mathcal{H}.$ We have
  $$d^2n\le \sum_{u\in V(\mathcal{H})}d(u)^2\le\sum_{u\in V(\mathcal{H})}\sum_{v\in N_1(u)}d(u)=\sum_{v\in V(\mathcal{H})}\sum_{u\in N_1(v)}d(u)\le3(t-1)^2(3t-1)n^2+2f(t)dn.$$
  Therefore, $d\le \sqrt{3(3t-1)}(t-1)n + O(t),$ where $O(t)$ means that it is a constant depending only on $t.$ Thus, $|B|=dn/3\le\frac{\sqrt{3(3t-1)}(t-1)}{3}n^{3/2}+O(n).$

Now the following theorem can be derived directly.
\begin{theorem}\label{theorem new result for K 2t}
Let $t\ge 3$ be an integer. We have that
$$ex_3(n,Tr(K_{2,t}))\le \left(\frac{\sqrt{3(3t-1)}(t-1)}{3}+\frac{\sqrt{t-1}}{2} \right)n^{3/2} + o(n^{3/2}).$$
\end{theorem}
\begin{remark}
 \begin{enumerate}
   \item By combining the conclusions in \cite{MR4245297} directly, one can also derive upper bounds for $ex_3(n, Tr(K_{2,t}))$ when $t$ is small, but our result, i.e., Theorem \ref{theorem new result for K 2t} is better.
   \item It was proved in \cite{MR4245297} that
$$
\lim_{t\rightarrow\infty}\lim_{n\rightarrow\infty}\frac{ex_3(n,Tr(K_{2,t}))}{t^{3/2}n^{3/2}}=\frac{1}{6},
$$
so, Theorem \ref{theorem new result for K 2t} works well only when $t$ is small.
 \end{enumerate}

\end{remark}

\section{Concluding remarks}
In this note, we give a new lower bound for $ex_r(n,Tr(K_{1,t}))$ by combinatorial coverings, and obtain optimal constructions for some cases. We also derive an improved upper bound for $ex_3(n, Tr(K_{2,t}))$ when $t$ is small. However, when $t$ is fixed, we still don't know the limit (if it exists) $\lim_{n\rightarrow\infty}ex_3(n,Tr(K_{2,t}))/n^{3/2}.$

\bibliographystyle{abbrv}
\bibliography{REF}
\end{document}